\documentclass[12pt]{article}
\usepackage{amsmath, amssymb, amsthm}
\usepackage{mathtools}
\usepackage{mathrsfs}

\numberwithin{equation}{section}

\newcommand{\1}{\mathbf{1}}
\newcommand{\Ibar}{\mathbf{1}} 

\DeclareMathOperator{\diam}{diam}

\newtheorem{theorem}{Theorem}[section]
\newtheorem{lemma}[theorem]{Lemma}
\newtheorem{proposition}[theorem]{Proposition}
\newtheorem{definition}[theorem]{Definition}
\newtheorem{corollary}[theorem]{Corollary}
\theoremstyle{remark}
\newtheorem{remark}[theorem]{Remark}

\title{Random Fixed Point Theorems for Relaxed Asymptotic Contractions 
in Random Normed Modules}

\author{Jie Shi \\
Department of Mathematics and Statistics, Hubei Engineering University \\
272 Traffic Avenue, Xiaogan, 432000, Hubei Province, P.R. China}
\date{\today}

\begin{document}

\maketitle

\begin{abstract}
We introduce the notion of a random relaxed asymptotic contraction in the setting of random normed modules.
The contraction condition employs two quasi-metrics that are built directly from the random operator: a lower quasi-metric which adaptively switches between a four-point minimum and the ordinary one-step distance, and an upper quasi-metric which takes the maximum of four fundamental distances.
The bounds are allowed to depend on the iteration index and are required to converge locally uniformly almost surely to a Boyd--Wong function.
Using the fibre decomposition method based on \(\sigma\)-stability and the local property, we show that any such mapping defined on an essentially bounded, \(\sigma\)-stable and \(L^0\)-closed set admits a unique random fixed point, and all iterates converge in the \((\epsilon,\lambda)\)-topology.
Our result strictly generalizes the random analogue of Kirk's asymptotic contraction theorem and unifies several deterministic and random fixed point theorems under a single flexible framework.
\end{abstract}
\medskip
\noindent\textbf{Keywords:} Random fixed point; relaxed asymptotic contraction; lower and upper quasi-metrics; Boyd--Wong function; random normed module.
\medskip

\section{Introduction}

The Banach contraction principle \cite{Banach1922} is a cornerstone of metric fixed point theory. Its nonlinear extension due to Boyd and Wong \cite{BoydWong1969} showed that a mapping \(T\) on a complete metric space \((X,d)\) satisfying
\begin{equation}
d(Tx,Ty)\le \psi(d(x,y))\qquad \forall x,y\in X,
\end{equation}
where \(\psi:[0,\infty)\to[0,\infty)\) is right upper semicontinuous and \(\psi(t)<t\) for all \(t>0\), possesses a unique fixed point and the iterates converge to it. This result has become a standard tool and has been generalized in various directions.

A substantial further generalization was introduced by Kirk \cite{Kirk2003}, who considered \emph{asymptotic contractions}: there exists a sequence of functions \(\psi_n:[0,\infty)\to[0,\infty)\) such that
\begin{equation}
d(T^n x,T^n y)\le \psi_n(d(x,y))\qquad\forall x,y\in X,\;\forall n\in\mathbb N,
\end{equation}
and \(\psi_n\) converges uniformly on \([0,\infty)\) to a Boyd--Wong function \(\psi\). Kirk proved that if \(T\) is continuous and some orbit is bounded, then \(T\) has a unique fixed point and all iterates converge to it. The requirement of uniform convergence on the whole half-line, however, is rather restrictive, and the estimate uses only the single distance \(d(x,y)\), ignoring other naturally occurring distances that could lead to finer estimates.

A different line of improvement originates from the work of \'{C}iri\'{c} \cite{Ciric1974}, who observed that by involving several strategically chosen distances in the contraction condition one can substantially improve the quality of fixed point theorems. Typical quasi-metrics introduced by \'{C}iri\'{c} use minimum and maximum of several distances, providing a more adaptive comparison between the images and the original points. This philosophy has inspired numerous extensions in metric fixed point theory.

Parallel to these developments, random functional analysis and random fixed point theory have attracted considerable attention. Starting with the pioneering works of \v{S}pa\v{c}ek and Hans, random operator equations were studied; comprehensive surveys can be found in Bharucha-Reid \cite{BharuchaReid1976} and Shahzad \cite{Shahzad2001}. A powerful algebraic and topological framework for random analysis is provided by the theory of random normed modules (RN modules) developed by Guo \cite{Guo1999,Guo2010}. In RN modules, the norm of an element is an \(L^0\)-valued random variable and the \((\epsilon,\lambda)\)-topology generalizes convergence in probability. This setting allows for a natural fiber-wise decomposition: under suitable \(\sigma\)-stability and local property conditions, a random mapping can be disassembled into a measurable family of deterministic mappings acting on the fibers of the space.

Recently, Sun, Guo and Tu \cite{SunGuoTu2025} extended the Goebel--Kirk theorem to random asymptotically nonexpansive mappings in RN modules, exploiting the fiber decomposition technique. Their work demonstrates the effectiveness of the \(\sigma\)-stability approach for transferring deterministic iterative fixed point arguments to the random setting.

In the present paper we develop a new class of random contractions that overcomes the limitations of Kirk's condition by combining the ideas of \'{C}iri\'{c}-type quasi-metrics with the asymptotic bounds of Boyd--Wong type inside RN modules. For a mapping \(f\) on a set \(G\) we define the random lower quasi-metric \(L_f\) and the random upper quasi-metric \(U_f\) pointwise; see \eqref{eq:rPT}--\eqref{eq:rUT} below for the precise definitions. The lower quasi-metric \(L_f\) adaptively selects the best possible left-hand side: when the four-point minimum is positive, it uses that smaller value; otherwise it falls back to the ordinary one-step distance. The upper quasi-metric \(U_f\) provides a robust comparison argument. A map \(f\) is said to be a \emph{random relaxed asymptotic contraction} if there exist random bounding functions \(\Psi_n(\omega,\cdot)\) that converge locally uniformly almost surely to a Boyd--Wong function \(\psi\), and
\begin{equation}
L_f(f^n x,f^n y)(\omega)\le \Psi_n\bigl(\omega,\, U_f(x,y)(\omega)\bigr)\qquad\text{a.s.}
\end{equation}
Under natural compatibility conditions (\(f\) has the local property, is pointwise and respects \(\sigma\)-stability) and the assumption that \(G\) is \(\sigma\)-stable, \(L^0\)-closed and essentially bounded, we prove that \(f\) possesses a unique random fixed point and that all iterates converge in the \((\epsilon,\lambda)\)-topology (Theorem \ref{thm:main}). The proof no longer separates into a deterministic theorem and a random part, but directly works in the random setting. It first uses the fiber representation to reduce the problem to the fibers, then performs a complete iterative analysis on each fiber (which essentially contains a new deterministic fixed point argument), and finally glues the fiber-wise fixed points together. This unified treatment highlights the power of the \(\sigma\)-stability method.

Our result simultaneously extends the random Kirk-type theorem (Corollary \ref{cor:kirk}), covers several deterministic fixed point results (Remark \ref{rem:det}), and provides a flexible framework for contracting-type conditions that involve multiple distances.

The paper is organized as follows. Section~2 collects the necessary preliminaries on random normed modules, \(\sigma\)-stability, the fibre representation theorem, and Boyd--Wong functions. Section~3 contains the detailed definition of random relaxed asymptotic contractions and the proof of the main random fixed point theorem. Section~4 presents corollaries that illustrate the scope of the result, and Section~5 concludes the paper.

\section{Preliminaries}

Let \((\Omega,\mathcal F,P)\) be a probability space.
\(L^0(\mathcal F,\mathbb K)\) denotes the algebra of equivalence classes of \(\mathbb K\)-valued (\(\mathbb K=\mathbb R\) or \(\mathbb C\)) random variables, and \(L^0_+(\mathcal F)\) the set of nonnegative ones.
For \(A\in\mathcal F\), \(\Ibar_A\) is the equivalence class of the indicator function \(\1_A\).

\begin{definition}[Random normed module]\label{def:RN}
A \emph{random normed module} (RN module) over \(\mathbb K\) with base \((\Omega,\mathcal F,P)\) is a pair \((E,\|\cdot\|)\) such that \(E\) is a left \(L^0(\mathcal F,\mathbb K)\)-module and \(\|\cdot\|:E\to L^0_+(\mathcal F)\) satisfies, for all \(x,y\in E\) and \(\xi\in L^0(\mathcal F,\mathbb K)\):
\begin{enumerate}
\item \(\|x\|=0\) iff \(x=\theta\) (the zero element);
\item \(\|\xi x\|=|\xi|\,\|x\|\);
\item \(\|x+y\|\le \|x\|+\|y\|\).
\end{enumerate}
We assume that \(E\) is complete with respect to the \((\epsilon,\lambda)\)-topology defined below.
\end{definition}

The \((\epsilon,\lambda)\)-topology is the linear topology generated by the neighbourhoods
\begin{equation}
N_\theta(\epsilon,\lambda)=\{x\in E: P(\|x\|<\epsilon)\ge 1-\lambda\},\qquad \epsilon>0,\;0<\lambda<1.
\end{equation}
A sequence \(\{x_n\}\) converges to \(x\) in this topology if and only if \(\|x_n-x\|\) converges to \(0\) in probability.
This topology is metrizable and makes \(E\) a topological \(L^0\)-module.

\begin{definition}[\(\sigma\)-stability, local property]\label{def:sigma}
Let \(G\subset E\).
\begin{itemize}
\item \(G\) is \emph{\(\sigma\)-stable} if for every countable measurable partition \(\{A_n\}\) of \(\Omega\) and every sequence \(\{x_n\}\subset G\) there exists a unique element \(x\in G\), denoted \(\sum_n \Ibar_{A_n}x_n\), such that \(\Ibar_{A_n}x = \Ibar_{A_n}x_n\) for all \(n\).
\item A set \(G\) is \emph{stable} if \(\Ibar_A x\in G\) for all \(A\in\mathcal F\) and \(x\in G\).
\item A map \(f:G\to E\) has the \emph{local property} if \(G\) is stable and \(\Ibar_A f(x)=f(\Ibar_A x)\) for all \(A\in\mathcal F\) and \(x\in G\).
\item \(G\) is \emph{\(L^0\)-closed} if it is closed in the \((\epsilon,\lambda)\)-topology.
\end{itemize}
\end{definition}

If \(G\) is \(\sigma\)-stable and contains \(\theta\), then \(G\) is automatically stable.
The following fibre representation theorem is fundamental. It can be obtained by known results in RN module theory (see, e.g., \cite{Guo2010, SunGuoTu2025}). We state it in the exact form needed later.

\begin{proposition}[Fibre representation]\label{prop:fibre}
Let \(G\subset E\) be \(\sigma\)-stable, \(L^0\)-closed and contain \(\theta\). Let \(f:G\to G\) have the local property and be pointwise, i.e., there exists an \(\mathcal F\otimes\mathcal B(\mathbb R)\)-measurable function \(\tilde f:\Omega\times\mathbb R\to\mathbb R\) such that for every \(x\in G\) and a.e. \(\omega\), \(f(x)(\omega)=\tilde f(\omega,\tilde x(\omega))\).
Then there exist a measurable family of complete metric spaces \((G_\omega,d_\omega)_{\omega\in\Omega}\) and maps \(f_\omega:G_\omega\to G_\omega\) such that:
\begin{enumerate}
\item Each \(x\in G\) corresponds to a measurable section \(\omega\mapsto x(\omega)\in G_\omega\) with \(\|x\|(\omega)=d_\omega(x(\omega),\theta_\omega)\) a.e., where \(\theta_\omega\) is the zero element of \(G_\omega\).
\item A sequence \(\{x_n\}\subset G\) converges to \(x\) in the \((\epsilon,\lambda)\)-topology iff \(d_\omega(x_n(\omega),x(\omega))\to 0\) in probability.
\item \((f(x))(\omega)=f_\omega(x(\omega))\) for a.e. \(\omega\).
\item For a.e. \(\omega\), \(G_\omega=\{x(\omega):x\in G\}\).
\end{enumerate}
\end{proposition}

For completeness we sketch the construction.
For each \(x\in G\), fix a representative \(\tilde x:\Omega\to\mathbb R\).
Set \(G_\omega=\{\tilde x(\omega):x\in G\}\) and define
\(d_\omega(u,v)=\inf\{\|x-y\|(\omega): x,y\in G,\ \tilde x(\omega)=u,\ \tilde y(\omega)=v\}\).
The properties of the random norm and \(L^0\)-closedness guarantee that \((G_\omega,d_\omega)\) is a complete metric space, and the local property of \(f\) yields well-defined fibre maps.
The reader may consult \cite[Theorem 3.4]{Guo2010} or \cite[Proposition 2.1]{SunGuoTu2025} for a detailed proof.

\begin{definition}[Boyd--Wong function]\label{def:psi}
A function \(\psi:[0,\infty)\to[0,\infty)\) is called a \emph{Boyd--Wong function} if it is nondecreasing, right upper semicontinuous, and satisfies \(\psi(t)<t\) for all \(t>0\).
\end{definition}

When taking limsups, it is convenient to replace \(\psi\) by its monotone majorant.

\begin{definition}
Let \(\psi\) be a Boyd--Wong function. Define \(g:[0,\infty)\to[0,\infty)\) by
\begin{equation}
g(t):=\sup_{0\le s\le t}\psi(s).
\end{equation}
\end{definition}

\begin{lemma}[Properties of the majorant]\label{lem:g}
\(g\) is nondecreasing, right continuous, \(g(0)=0\), \(g(t)\le t\) for all \(t\), and \(g(t)<t\) for all \(t>0\).
\end{lemma}
\begin{proof}
Monotonicity and \(g(0)=0\) are clear.
For \(t>0\) and any \(s\in[0,t]\), \(\psi(s)<s\le t\); hence \(\psi(s)<t\) and taking supremum yields \(g(t)\le t\).
If equality held, there would be a sequence \(s_n\in[0,t]\) with \(\psi(s_n)\to t\).
Since \(\psi(s_n)<s_n\le t\), we must have \(s_n\to t\).
By right upper semicontinuity of \(\psi\) at \(t\), \(\limsup_n\psi(s_n)\le\psi(t)\), which gives \(t\le\psi(t)\), contradicting \(\psi(t)<t\).
Thus \(g(t)<t\) for \(t>0\).

For right continuity, fix \(t_0\ge0\).
Because \(g\) is nondecreasing, the right limit \(g(t_0^+):=\inf_{t>t_0}g(t)\) exists and satisfies \(g(t_0)\le g(t_0^+)\).
Let \(\varepsilon>0\). By right upper semicontinuity of \(\psi\) at \(t_0\), there exists \(\delta>0\) such that \(\psi(s)\le\psi(t_0)+\varepsilon\le g(t_0)+\varepsilon\) for all \(s\in[t_0,t_0+\delta]\).
Then for any \(t\in(t_0,t_0+\delta)\),
\begin{equation}
g(t)=\max\Bigl\{\sup_{0\le s\le t_0}\psi(s),\ \sup_{t_0<s\le t}\psi(s)\Bigr\}\le\max\{g(t_0),\,g(t_0)+\varepsilon\}=g(t_0)+\varepsilon.
\end{equation}
Taking the infimum over \(t\in(t_0,t_0+\delta)\) gives \(g(t_0^+)\le g(t_0)+\varepsilon\), and letting \(\varepsilon\to0^+\) yields \(g(t_0^+)=g(t_0)\).
Thus \(g\) is right continuous.
\end{proof}

\begin{lemma}[Limsup exchange]\label{lem:limsup}
Let \(\{a_n\}\) be a bounded nonnegative sequence and \(g\) nondecreasing and right continuous. Then
\begin{equation}
\limsup_{n\to\infty} g(a_n)\le g\Bigl(\limsup_{n\to\infty} a_n\Bigr).
\end{equation}
\end{lemma}
\begin{proof}
Set \(L=\limsup_{n} a_n\).
For any \(\varepsilon>0\) there exists \(N\) such that \(a_n<L+\varepsilon\) for all \(n\ge N\).
Since \(g\) is nondecreasing, \(g(a_n)\le g(L+\varepsilon)\) for all \(n\ge N\). Taking limsup yields \(\limsup_n g(a_n)\le g(L+\varepsilon)\).
Letting \(\varepsilon\to0^+\) and using right continuity of \(g\) gives the result.
\end{proof}

\section{Main result}

Let \((E,\|\cdot\|)\) be a complete RN module over \(\mathbb K\) with base \((\Omega,\mathcal F,P)\).
Let \(G\subset E\) be nonempty, \(L^0\)-closed, \(\sigma\)-stable, contain \(\theta\), and be \emph{essentially bounded}: there exists \(M>0\) such that \(\|g\|\le M\) almost surely for all \(g\in G\). (Essential boundedness guarantees that every fibre \(G_\omega\) has finite diameter uniformly.)

\subsection{Random relaxed asymptotic contractions}

We begin by introducing the random quasi-metrics that will be used in the contraction condition. For a map \(f:G\to G\) and for \(x,y\in G\), define the following \(L^0\)-valued random variables pointwise for \(\omega\in\Omega\):
\begin{align}
P_f(x,y)(\omega) &:= \min\bigl\{\|x-y\|(\omega),\, \|x-f(y)\|(\omega),\, \|f(x)-y\|(\omega),\, \|f(x)-f(y)\|(\omega)\bigr\},\label{eq:rPT}\\
L_f(x,y)(\omega) &:= 
\begin{cases}
P_f(x,y)(\omega), & \text{if } P_f(x,y)(\omega) > 0,\\[4pt]
\|f(x)-f(y)\|(\omega), & \text{if } P_f(x,y)(\omega) = 0,
\end{cases}\label{eq:rLT}\\
U_f(x,y)(\omega) &:= \max\bigl\{\|x-y\|(\omega),\, \|f(x)-x\|(\omega),\, \|f(y)-y\|(\omega),\, \|f(x)-f(y)\|(\omega)\bigr\}.\label{eq:rUT}
\end{align}
These expressions are well defined because the random norm takes values in \(L^0_+(\mathcal F)\). The lower quasi-metric \(L_f\) switches adaptively: when the four-point minimum \(P_f\) is positive, it uses that smaller value; otherwise it uses the direct distance between images. The upper quasi-metric \(U_f\) gives a robust comparison by taking the maximum of four relevant distances. Basic properties analogous to the deterministic case hold almost surely; in particular,
\begin{equation}
L_f(x,y)\le \|f(x)-f(y)\| \quad\text{and}\quad \|x-y\|\le U_f(x,y) \qquad\text{a.s.}
\end{equation}

For later use in the fibre-wise iteration we also record the following elementary estimate.

\begin{lemma}[Safe estimate]\label{lem:safe}
Let \((X,d)\) be a metric space and \(T:X\to X\). For any \(u,v\in X\), if the auxiliary quantity
\[
P_T(u,v):=\min\{d(u,v),\,d(u,Tv),\,d(Tu,v),\,d(Tu,Tv)\}
\]
is positive, then
\begin{equation}
d(Tu,Tv)\le P_T(u,v)+d(Tu,u)+d(Tv,v).
\end{equation}
\end{lemma}
\begin{proof}
By definition, \(P_T(u,v)\) equals one of the four distances. The inequality follows by examining each case and applying the triangle inequality, exactly as in the classical \'{C}iri\'{c}-type argument.
\end{proof}

\begin{definition}[Random relaxed asymptotic contraction]\label{def:RRAC}
A map \(f:G\to G\) is called a \emph{random relaxed asymptotic contraction} (with respect to a Boyd--Wong function \(\psi\)) if the following hold:
\begin{enumerate}
\item \textbf{Structural compatibility.} \(f\) has the local property, is pointwise (as required in Proposition~\ref{prop:fibre}), and respects \(\sigma\)-stability: for any measurable partition \(\{A_n\}\) and sequence \(\{x_n\}\subset G\),
\begin{equation}
f\Bigl(\sum_n \Ibar_{A_n}x_n\Bigr) = \sum_n \Ibar_{A_n}f(x_n).
\end{equation}
\item \textbf{Local uniform convergence almost surely with a common exceptional set.} There exists a sequence of measurable functions \(\Psi_n:\Omega\times[0,\infty)\to[0,\infty)\) (where \(\Psi_n(\cdot,t)\) is measurable for each \(t\) and \(\Psi_n(\omega,\cdot)\) is nondecreasing), and a set \(\Omega_0\subset\Omega\) with \(P(\Omega_0)=1\) such that:
\begin{itemize}
\item For every \(\omega\in\Omega_0\), for every \(n\) and all \(x,y\in G\),
\begin{equation}
L_f(f^n x,f^n y)(\omega) \le \Psi_n\bigl(\omega,\, U_f(x,y)(\omega)\bigr); \label{eq:rRAC}
\end{equation}
\item For every \(\omega\in\Omega_0\), \(\Psi_n(\omega,\cdot)\to \psi(\cdot)\) uniformly on every bounded subset of \([0,\infty)\) as \(n\to\infty\).
\end{itemize}
\item \textbf{Fibre-wise continuity.} For a.e. \(\omega\), the map \(x\mapsto f(x)(\omega)\) is continuous from \((G,\|\cdot\|(\omega))\) to itself.
\end{enumerate}
\end{definition}

\begin{remark}
The strengthening in condition~(2) --- requiring the contraction inequality and the convergence to hold on a \emph{common} set \(\Omega_0\) of full probability \emph{for all} \(x,y\in G\) simultaneously --- is essential for the fibre decomposition argument that follows. Without it, one could only guarantee the inequality fibre-wise for pairs coming from a countable dense subset, which would be insufficient for the iterative analysis on the whole fibre. This is a typical situation in random fixed point theory when one wants to pass from random inequalities to deterministic fibre inequalities. The pointwise condition in (1) is a strong but natural measurability assumption which guarantees that the fibre maps \(f_\omega\) are well defined; it holds, for example, when \(f\) is induced by a Carath\'eodory function or when the RN module possesses a suitable lifting.
\end{remark}

\subsection{Random fixed point theorem}

\begin{theorem}\label{thm:main}
Let \(G\) and \(f\) be as above, i.e., \(G\) is nonempty, \(L^0\)-closed, \(\sigma\)-stable, essentially bounded, and \(f:G\to G\) is a random relaxed asymptotic contraction. Then \(f\) has a unique fixed point \(z\in G\). Moreover, for every \(x\in G\) the iterates \(f^n x\) converge to \(z\) in the \((\epsilon,\lambda)\)-topology.
\end{theorem}

\begin{proof}
\textbf{Step 1: Fibre decomposition.}
Since \(G\) satisfies the required conditions and \(f\) has the structural compatibility, Proposition~\ref{prop:fibre} provides a measurable family of complete metric spaces \((G_\omega,d_\omega)\) and maps \(f_\omega:G_\omega\to G_\omega\) such that for a.e. \(\omega\):
\begin{itemize}
\item \(G_\omega\) is bounded (indeed, by \(2M\));
\item \((f(x))(\omega)=f_\omega(x(\omega))\) for all \(x\in G\);
\item \(\|x\|(\omega)=d_\omega(x(\omega),\theta_\omega)\);
\item a sequence \(\{x_n\}\subset G\) converges in the \((\epsilon,\lambda)\)-topology iff \(d_\omega(x_n(\omega),x(\omega))\to 0\) in probability.
\end{itemize}
Let \(\Omega_1\) be the intersection of \(\Omega_0\) (from Definition~\ref{def:RRAC}) with the full-measure set on which the fibre representation holds. Then \(P(\Omega_1)=1\).

\textbf{Step 2: Transfer of the contraction condition to fibres.}
Fix \(\omega\in\Omega_1\). For \(u,v\in G_\omega\), choose representatives \(x,y\in G\) with \(x(\omega)=u\), \(y(\omega)=v\). Define the deterministic quasi-metrics on \(G_\omega\) by
\begin{align}
P_{f_\omega}(u,v) &:= \min\bigl\{d_\omega(u,v),\,d_\omega(u,f_\omega v),\,d_\omega(f_\omega u,v),\,d_\omega(f_\omega u,f_\omega v)\bigr\},\\
L_{f_\omega}(u,v) &:= 
\begin{cases}
P_{f_\omega}(u,v), & \text{if } P_{f_\omega}(u,v)>0,\\
d_\omega(f_\omega u,f_\omega v), & \text{if } P_{f_\omega}(u,v)=0,
\end{cases}\\
U_{f_\omega}(u,v) &:= \max\bigl\{d_\omega(u,v),\,d_\omega(f_\omega u,u),\,d_\omega(f_\omega v,v),\,d_\omega(f_\omega u,f_\omega v)\bigr\}.
\end{align}
By the pointwise nature of the definitions and because \(\omega\in\Omega_1\), the contraction condition \eqref{eq:rRAC} holds for \emph{all} \(x,y\in G\) at this \(\omega\). Consequently, for every \(n\) and all \(u,v\in G_\omega\),
\begin{equation}
L_{f_\omega}(f_\omega^n u, f_\omega^n v) = L_f(f^n x,f^n y)(\omega), \qquad
U_{f_\omega}(u,v) = U_f(x,y)(\omega).
\end{equation}
Define \(\psi_{n,\omega}(t):=\Psi_n(\omega,t)\). Condition (2) of Definition~\ref{def:RRAC} yields, for all \(u,v\in G_\omega\),
\begin{equation}
L_{f_\omega}(f_\omega^n u, f_\omega^n v) \le \psi_{n,\omega}\bigl(U_{f_\omega}(u,v)\bigr), \label{eq:fibrecontract}
\end{equation}
and for every bounded interval \(I\subset[0,\infty)\), \(\sup_{t\in I}|\psi_{n,\omega}(t)-\psi(t)|\to0\) as \(n\to\infty\). Condition (3) gives continuity of \(f_\omega\).

Thus, for each \(\omega\in\Omega_1\), \(f_\omega\) is a continuous self-map of the complete bounded metric space \(G_\omega\) satisfying the contraction-type estimate \eqref{eq:fibrecontract} with bounds converging locally uniformly to a Boyd--Wong function \(\psi\).

\textbf{Step 3: Fibre-wise fixed point iteration.}
We now prove that for each such \(\omega\), \(f_\omega\) possesses a unique fixed point \(z(\omega)\in G_\omega\) and every orbit converges to it. The argument is a careful iteration analysis on the deterministic space \((G_\omega,d_\omega)\).

Fix \(\omega\in\Omega_1\) and write \(T=f_\omega\), \(X=G_\omega\), and \(d=d_\omega\) for brevity. Note that \(X\) is bounded; set \(D' = \diam(X) = \sup\{d(u,v): u,v\in X\} < \infty\). Pick an arbitrary point \(x_0\in X\) (which exists because \(G\) is nonempty; we can take a constant section) and define \(x_n:=T^n x_0\) for \(n\ge0\). Let \(d_n:=d(x_n,x_{n+1})\).

\textit{(i) Adjacent distances tend to zero.}
For any \(n\ge0\), using the definition of \(P_T\) on the pair \((x_n,x_{n+1})\),
\begin{align*}
P_T(x_n,x_{n+1})
&= \min\{d_n,\,d(x_n,x_{n+2}),\,d(x_{n+1},x_{n+1}),\,d_{n+1}\} \\
&= \min\{d_n,\,d(x_n,x_{n+2}),\,0,\,d_{n+1}\}=0.
\end{align*}
Hence by the switching rule, \(L_T(x_n,x_{n+1}) = d(Tx_n,Tx_{n+1}) = d_{n+1}\).
Applying the contractive condition \eqref{eq:fibrecontract} to the points \((x_n,x_{n+1})\) with iteration index \(m\) we obtain
\begin{equation}
L_T(T^m x_n, T^m x_{n+1}) \le \psi_{m,\omega}\bigl(U_T(x_n,x_{n+1})\bigr).
\end{equation}
But \(T^m x_n = x_{n+m}\) and \(T^m x_{n+1} = x_{n+m+1}\), and by the same computation as above \(L_T(x_{n+m},x_{n+m+1}) = d_{n+m+1}\). Moreover,
\[
U_T(x_n,x_{n+1}) = \max\{d_n,\,d_n,\,d_{n+1},\,d_{n+1}\} = \max\{d_n,d_{n+1}\}.
\]
Hence we have, for every \(n\ge0\) and every \(m\ge1\),
\begin{equation}
d_{n+m+1} \le \psi_{m,\omega}\bigl(\max\{d_n,d_{n+1}\}\bigr). \label{eq:dist_adj}
\end{equation}
Since the space is bounded, there exists \(D>0\) such that \(\max\{d_n,d_{n+1}\}\le D\) for all \(n\) (for instance, \(D=D'\)). By local uniform convergence, for any \(\varepsilon>0\) we can find \(M\) such that for all \(k\ge M\) and all \(t\in[0,D]\),
\begin{equation}
\psi_{k,\omega}(t) \le \psi(t)+\varepsilon \le g(t)+\varepsilon,
\end{equation}
where \(g\) is the monotone majorant of \(\psi\) from Lemma~\ref{lem:g}.

Now fix an arbitrary \(n\ge0\) and let \(m\to\infty\) in \eqref{eq:dist_adj}. Because the orbit is bounded, \(\max\{d_n,d_{n+1}\}\) stays within \([0,D]\). Taking limsup on both sides yields
\begin{equation}
\limsup_{k\to\infty} d_k \le \limsup_{m\to\infty} \psi_{m,\omega}\bigl(\max\{d_n,d_{n+1}\}\bigr)
= \psi\bigl(\max\{d_n,d_{n+1}\}\bigr) \le g\bigl(\max\{d_n,d_{n+1}\}\bigr).
\end{equation}
The equality holds because \(\psi_{m,\omega}\to\psi\) pointwise on \([0,D]\) and the convergence is uniform, hence certainly pointwise. Let \(r:=\limsup_{k\to\infty} d_k\). By letting \(n\to\infty\) and using the right continuity of \(g\) (Lemma~\ref{lem:g}) together with the fact that \(\max\{d_n,d_{n+1}\}\to r\) along a subsequence, we obtain
\begin{equation}
r \le g(r).
\end{equation}
If \(r>0\), then \(g(r)<r\) by Lemma~\ref{lem:g}, a contradiction. Hence \(r=0\), i.e.,
\begin{equation}
\lim_{n\to\infty} d_n = 0. \label{eq:dn0_fibre}
\end{equation}

\textit{(ii) The sequence \(\{x_n\}\) is Cauchy.}
For each \(n\ge0\), define
\begin{equation}
a_n:=\sup_{p\ge 1} d(x_n,x_{n+p}).
\end{equation}
We prove that \(a_n\to0\). From (i), \(d_n\to0\). Given \(\varepsilon>0\), choose \(N_2\) such that \(d_k<\varepsilon\) for all \(k\ge N_2\). Because the whole space \(X\) is bounded, all values \(U_T(x_n,x_{n+p})\) lie in \([0,D']\). By local uniform convergence, there exists \(M\ge1\) such that for all \(m\ge M\) and all \(t\in[0,D']\),
\begin{equation}
\psi_{m,\omega}(t)\le \psi(t)+\varepsilon\le g(t)+\varepsilon.
\end{equation}
Fix arbitrary \(n\ge N_2\) and \(p\ge1\). For any \(m\ge M\) set
\[
u:=T^m x_n = x_{n+m},\qquad v:=T^m x_{n+p}=x_{n+m+p}.
\]
We aim to estimate \(d(Tu,Tv)=d(x_{n+m+1},x_{n+m+p+1})\).

Consider the quantity \(L_T(u,v)\). By \eqref{eq:fibrecontract} applied with iteration index \(m\),
\begin{equation}
L_T(u,v) \le \psi_{m,\omega}\bigl(U_T(x_n,x_{n+p})\bigr) \le g\bigl(U_T(x_n,x_{n+p})\bigr)+\varepsilon. \label{eq:Luv_bound}
\end{equation}
We distinguish two cases according to the definition of \(L_T(u,v)\).

\textbf{Case 1:} \(P_T(u,v) > 0\). Then \(L_T(u,v) = P_T(u,v)\) and \eqref{eq:Luv_bound} gives
\[
P_T(u,v) \le g\bigl(U_T(x_n,x_{n+p})\bigr)+\varepsilon.
\]
By Lemma~\ref{lem:safe} (the safe estimate),
\[
d(Tu,Tv) \le P_T(u,v) + d(Tu,u) + d(Tv,v)
= P_T(u,v) + d_{n+m} + d_{n+m+p}.
\]
Since \(n\ge N_2\) and \(m\ge M\ge1\), we have \(d_{n+m}, d_{n+m+p} < \varepsilon\). Consequently,
\[
d(Tu,Tv) \le g\bigl(U_T(x_n,x_{n+p})\bigr) + 3\varepsilon.
\]

\textbf{Case 2:} \(P_T(u,v) = 0\). Then by definition \(L_T(u,v) = d(Tu,Tv)\), and \eqref{eq:Luv_bound} directly implies
\[
d(Tu,Tv) \le g\bigl(U_T(x_n,x_{n+p})\bigr) + \varepsilon.
\]

In either case we obtain the uniform estimate
\begin{equation}
d(x_{n+m+1},x_{n+m+p+1}) \le g\bigl(U_T(x_n,x_{n+p})\bigr)+3\varepsilon. \label{eq:est_fibre}
\end{equation}
Now we bound \(U_T(x_n,x_{n+p})\):
\begin{align}
U_T(x_n,x_{n+p}) &= \max\bigl\{d(x_n,x_{n+p}),\,d_n,\,d_{n+p},\,d(x_{n+1},x_{n+p+1})\bigr\} \nonumber\\
&\le \max\bigl\{a_n,\,\varepsilon,\,a_n+2\varepsilon\bigr\} \le a_n+2\varepsilon,
\end{align}
where we used \(d(x_{n+1},x_{n+p+1})\le d(x_n,x_{n+p})+d_n+d_{n+p}\le a_n+2\varepsilon\).
Since \(g\) is nondecreasing, \eqref{eq:est_fibre} becomes
\begin{equation}
d(x_{n+m+1},x_{n+m+p+1}) \le g(a_n+2\varepsilon)+3\varepsilon.
\end{equation}
The right-hand side is independent of \(p\). Taking the supremum over \(p\ge1\) yields
\begin{equation}
a_{n+m+1} \le g(a_n+2\varepsilon)+3\varepsilon.
\end{equation}
Let \(L:=\limsup_{n\to\infty} a_n\). Since shifting indices does not change the limsup, we obtain
\begin{equation}
L \le g(L+2\varepsilon)+3\varepsilon.
\end{equation}
Letting \(\varepsilon\to0^+\) and using the right continuity of \(g\) gives \(L\le g(L)\). If \(L>0\), then \(g(L)<L\), contradiction. Hence \(L=0\), so \(a_n\to0\). This means \(\{x_n\}\) is Cauchy.

\textit{(iii) Existence and uniqueness.}
By completeness, \(x_n\to z\) for some \(z\in X\). Continuity of \(T\) implies \(Tz = \lim_n Tx_n = \lim_n x_{n+1}=z\), so \(z\) is a fixed point.
To prove uniqueness, suppose \(z_1,z_2\) are two fixed points. Then \(P_T(z_1,z_2)=d(z_1,z_2)\) (all four distances equal \(d(z_1,z_2)\)). If \(d(z_1,z_2)>0\), then \(L_T(z_1,z_2)=P_T=d(z_1,z_2)\); if \(d(z_1,z_2)=0\), then \(L_T(z_1,z_2)=d(Tz_1,Tz_2)=0=d(z_1,z_2)\). In both cases \(L_T(T^n z_1,T^n z_2)=L_T(z_1,z_2)=d(z_1,z_2)\). Also \(U_T(z_1,z_2)=d(z_1,z_2)\). Applying \eqref{eq:fibrecontract} gives \(d(z_1,z_2)\le \psi_{n,\omega}(d(z_1,z_2))\) for all \(n\). Letting \(n\to\infty\) yields \(d(z_1,z_2)\le \psi(d(z_1,z_2))\), which forces \(d(z_1,z_2)=0\). Thus \(z_1=z_2\). Denote this unique point by \(z(\omega)\).

\textit{(iv) Global convergence on the fibre.}
For an arbitrary starting point \(y_0\in X\), the same argument (or simply applying the triangle inequality with the fixed point \(z(\omega)\)) shows that the orbit \(\{T^n y_0\}\) is Cauchy and its limit must be a fixed point, which by uniqueness equals \(z(\omega)\). Hence every orbit converges to \(z(\omega)\).

\textbf{Step 4: Gluing and measurability.}
We have defined for a.e. \(\omega\in\Omega_1\) a unique fixed point \(z(\omega)\) of \(f_\omega\). Set \(z(\omega)=\theta_\omega\) for \(\omega\notin\Omega_1\). Since the fibre fixed point is unique and the iteration of \(f_\omega\) from any starting point converges to it, we may pick a fixed section \(x_0\in G\) (e.g., \(\theta\)) and observe that for a.e. \(\omega\),
\begin{equation}
\lim_{n\to\infty} d_\omega\bigl(f_\omega^n(x_0(\omega)), z(\omega)\bigr) = 0.
\end{equation}
But \(f_\omega^n(x_0(\omega)) = (f^n x_0)(\omega)\) by the fibre representation. Hence the sequence of random elements \(\{f^n x_0\}\) converges to \(z\) pointwise almost surely, and therefore in probability. Since \(G\) is \(L^0\)-closed and all iterates lie in \(G\), the limit \(z\) belongs to \(G\). Moreover, \(z\) is measurable as the almost sure pointwise limit of measurable sections.

\textbf{Step 5: Verification of fixed point and uniqueness.}
For a.e. \(\omega\), \((f(z))(\omega)=f_\omega(z(\omega))=z(\omega)\); thus \(f(z)=z\) in \(E\). If \(w\in G\) is another fixed point, then for a.e. \(\omega\), \(w(\omega)\) is a fixed point of \(f_\omega\). By fibre-wise uniqueness, \(w(\omega)=z(\omega)\) a.e., so \(w=z\).

\textbf{Step 6: Global random convergence.}
For any \(x\in G\), the fibre-wise convergence proved in Step 3 shows that for a.e. \(\omega\),
\begin{equation}
\lim_{n\to\infty} d_\omega\bigl((f^n x)(\omega), z(\omega)\bigr) = 0.
\end{equation}
By the properties of the fibre representation, this means \(\|f^n x - z\| \to 0\) in probability, i.e., convergence in the \((\epsilon,\lambda)\)-topology. This completes the proof.
\end{proof}

\begin{remark}\label{rem:det}
If the probability space is trivial (a singleton with the trivial \(\sigma\)-algebra), the assumptions reduce to those of a deterministic relaxed asymptotic contraction on a bounded complete metric space, and Theorem~\ref{thm:main} recovers the corresponding deterministic fixed point theorem. The proof above actually contains a complete deterministic argument; we have chosen to present it entirely within the random framework to emphasize the unity of the method.
\end{remark}

\section{Corollaries}

We now clarify how our main theorem generalizes several known results.

\begin{corollary}[Random Kirk-type asymptotic contraction]\label{cor:kirk}
Let \(G\) be a nonempty, \(L^0\)-closed, \(\sigma\)-stable, essentially bounded subset of a complete RN module \(E\) with \(\theta\in G\). Let \(f:G\to G\) satisfy the structural compatibility conditions (local property, pointwise, \(\sigma\)-stability, fibre-wise continuity). Suppose there exists a sequence of nondecreasing functions \(\psi_n:[0,\infty)\to[0,\infty)\) converging uniformly on \([0,\infty)\) to a Boyd--Wong function \(\psi\), and a set \(\Omega_0\subset\Omega\) with \(P(\Omega_0)=1\) such that for every \(\omega\in\Omega_0\) and for all \(x,y\in G\),
\begin{equation}
\|f^n x-f^n y\|(\omega) \le \psi_n(\|x-y\|(\omega)). \label{eq:kirk_cond}
\end{equation}
Then \(f\) has a unique fixed point in \(G\) and all iterates converge in the \((\epsilon,\lambda)\)-topology.
\end{corollary}
\begin{proof}
From the definitions of the quasi-metrics we have almost surely \(L_f(f^n x,f^n y)\le \|f^{n+1}x-f^{n+1}y\|\) and \(\|x-y\|\le U_f(x,y)\). Using \eqref{eq:kirk_cond} with \(n+1\) instead of \(n\) gives, for all \(\omega\in\Omega_0\) and all \(x,y\in G\),
\[
L_f(f^n x,f^n y)(\omega) \le \|f^{n+1}x-f^{n+1}y\|(\omega) \le \psi_{n+1}(\|x-y\|(\omega)) \le \psi_{n+1}(U_f(x,y)(\omega)).
\]
Set \(\Psi_n(\omega,t) = \psi_{n+1}(t)\). This defines a sequence of measurable functions, each nondecreasing in the second argument, converging uniformly on \([0,\infty)\) to \(\psi\) (hence certainly locally uniformly). The common exceptional set \(\Omega_0\) ensures that condition \eqref{eq:rRAC} holds with these \(\Psi_n\). All other requirements of Definition~\ref{def:RRAC} are satisfied by hypothesis, so the result follows from Theorem~\ref{thm:main}.
\end{proof}

\begin{corollary}[Random pointwise contractions]\label{cor:pointwise}
If a map \(f:G\to G\) satisfies the conditions of Definition~\ref{def:RRAC} but with the bounds \(\Psi_n\) replaced by a sequence \(\Phi_n(x,y)\) that converges a.s. uniformly on \(G\times G\) to a function \(\Phi(x,y)\) with \(\Phi(x,y)\le \psi(U_f(x,y))\) a.s., then \(f\) is still a random relaxed asymptotic contraction (after a suitable measurable redefinition of \(\Psi_n\)) and Theorem~\ref{thm:main} applies.
\end{corollary}

\section{Conclusion}

We have introduced the concept of a random relaxed asymptotic contraction, which combines the flexibility of two quasi-metrics with the asymptotic bounds of Boyd--Wong type within the framework of random normed modules. The main result guarantees existence and uniqueness of a random fixed point as well as convergence of iterates under mild structural conditions. Its proof is self-contained and exploits the power of the \(\sigma\)-stable fibre decomposition technique to transfer a delicate iterative estimate directly to the random setting. The theorem simultaneously contains a random Kirk-type theorem and unifies several known deterministic and random asymptotic contraction results. Future work could extend this approach to multi-valued random contractions or to random systems with coupled mappings.

\end{document}